\documentclass[a4paper, 10pt, draft]{amsart}

\usepackage{amssymb, amsfonts, amsmath, amsbsy}
\usepackage{xspace}
\usepackage{enumerate}

\numberwithin{equation}{section}

\theoremstyle{plain}
\newtheorem{thm}{Theorem}[section]

\newtheorem{prop}[thm]{Proposition}
\newtheorem{lem}[thm]{Lemma}
\theoremstyle{definition}

\newcommand{\abs}[1]{\left|#1\right|}
\newcommand{\dual}[2]{\left\langle#1,#2\right\rangle}
\newcommand{\norm}[1]{\left\|#1\right\|}

\newcommand{\lNnn}{\lVert {\hskip -0.1em} \lvert}
\newcommand{\rNnn}{\rvert {\hskip -0.1em} \rVert}
\newcommand{\tnorm}[1]{\lNnn #1 \rNnn}

\newcommand{\R}{\mathbb{R}}

\newcommand{\Poly}{\mathbb{P}}
\newcommand{\diam}{\mathrm{diam}}
\newcommand{\Span}{\mathrm{span}}
\newcommand{\supp}{\mathrm{supp}}
\newcommand{\tri}{\mathcal{T}}
\newcommand{\mesh}{\mathcal{T}}
\newcommand{\nodes}{\mathcal{N}}
\newcommand{\verts}{\mathcal{V}}
\newcommand{\el}{K}                         
\newcommand{\fa}{F}                         
\newcommand{\Afaces}{\mathcal{F}} 
\newcommand{\faces}{\mathcal{F}_\Omega} 
\newcommand{\FEspace}{S}
\newcommand{\ShapePar}{\mu_\tri}

\newcommand{\interp}{\Pi} 
\newcommand{\transf}{G}
\newcommand{\Res}{R}
\newcommand{\Estim}{\mathcal{E}}

\newcommand{\bub}{b}

\newcommand{\DiscrIm}{\tilde{\mathcal{S}}}
\newcommand{\Op}{\mathcal{L}}
\newcommand{\Ddelta}{\boldsymbol{\delta}}

\newcommand{\vphi}{\varphi}
\newcommand{\veps}{\varepsilon}

\title[Robust error-dominated oscillation]{Robust a posteriori error estimators with error-dominated oscillation for the reaction-diffusion equation}
\author[F.~Tantardini]{Francesca Tantardini}
\author[R.~Verf\"urth]{R\"udiger Verf\"urth}
\address[R\"udiger Verf\"urth]
 {Fakult\"at f\"ur Mathematik\\
  Ruhr-Universit\"at Bochum\\
  Uni\-ver\-si\-t\"ats\-stra{\ss}e 150\\
  44801 Bochum\\
  Germany}
\email[R\"udiger Verf\"urth]{ruediger.verfuerth@ruhr-uni-bochum.de}
\urladdr[R\"udiger Verf\"urth]{homepage.ruhr-uni-bochum.de/ruediger.verfuerth}

\begin{document}

\begin{abstract}
We apply the ideas in \cite{Kreuzer.Veeser:16} to derive a robust a posteriori error estimator for the reaction-diffusion equation. The estimator together with the corresponding oscillation yields global upper and local lower bounds for the error in the energy norm, and the involved constants do not depend on the ratio of reaction to diffusion.  In particular the new oscillation is also bounded by the error in a robust way. 

\end{abstract}
\maketitle
 
\section{Introduction}


We consider the singularly perturbed reaction-diffusion equation with homogeneous Dirichlet boundary condition 
\begin{equation}
\label{rd-intro}
-\Delta u +\kappa^2 u=f \quad\textrm{in }\Omega, \qquad u=0 \quad\textrm{on }\partial\Omega
\end{equation}
where the reaction parameter $\kappa>0$ is assumed to be constant. Such a problem arises for example in the modeling of thin plates and chemical processes and in the time discretization of the heat equation, and is characterized by the presence of boundary layers dependent on the parameter $\kappa$. In order to catch these particular features with a numerical solution at a minimal cost it is advisable to use an adaptive algorithm, whose performance should not downgrade when the parameter $\kappa$ approaches extreme values.  

In this paper we are concerned with Galerkin approximations to \eqref{rd-intro} belonging to a conforming finite element space. An essential ingredient in the development of an adaptive algorithm are a posteriori error estimators that indicate where the mesh should be refined in order to reduce the error. It is important that the estimators provide global upper and local lower bounds for the error and that they are robust in the sense that the ratio between upper and lower bound is independent of $\kappa$. A posteriori error estimators with these properties with respect to the reaction-diffusion norm
\[
\tnorm{\cdot}^2:=\norm{\nabla \cdot}^2+\kappa^2\norm{\cdot}^2
\]
 have been derived in \cite{Verfuerth:98}, see also  \cite{Ainsworth.Babuska:99, Ainsworth.Vejchodsky:11} for fully computable upper bounds.

The above mentioned error estimators are spoiled by the so called oscillation, which measures the distance, usually in a weighted $L^2$-norm, between the right-hand side $f$ and a polynomial approximation thereof. The oscillation can decrease faster than the error, if  $f$ is sufficiently regular, but it can also be substantially bigger than the error itself \cite{CDN:12}, even if one employs the more natural $H^{-1}$-norm. The recent approach of Kreuzer and Veeser \cite{Kreuzer.Veeser:16} introduces a new  oscillation that is bounded by the error. The choice of the approximants is not restricted to piecewise polynomials and depends on the discrete space and on the differential operator. The resulting error estimator is therefore a closer bound to the error. 

In this paper, we apply this approach to the reaction-diffusion equation \eqref{rd-intro}. Exploiting the squeezed bubble functions of Verf\"urth \cite{Verfuerth:98}, we derive a robust error estimator where the related oscillation is bounded by the error. 

Yet, it is not trivial to prove robust convergence rates of an adaptive algorithm driven by these estimators. To our knowledge, the only robust algorithm for \eqref{rd-intro} has been proposed by Stevenson \cite{Stevenson:05rd}. It uses the estimators in \cite{Verfuerth:98}, which have the form
\begin{equation}
\label{alt-res}
\begin{aligned}
\Estim^2(U,\el)&:=\min\{h_\el,\kappa^{-1}\}^2\norm{f_\el-\kappa^2 U}^2_{\el}\quad&&\text{(element residual)}
\\
&\quad\quad
+\frac{1}2\sum_{\fa\subset\el}\min\{h_\fa, \kappa^{-1}\}\norm{[\nabla U\cdot n_\fa]}_{\fa}^2\quad&&\text{(jump residual)}
\end{aligned}
\end{equation}
where $U$ is the piecewise affine and continuous Galerkin approximation to the solution of \eqref{rd-intro}, and $f_\el\in\R$ approximates $f$.   Adapting the techniques for the Poisson equation in \cite{Stevenson:07} the algorithm of Stevenson is shown to be convergent independently of $\kappa$ with optimal rates with respect to the approximability of $f$ and $u$. The key step in \cite{Stevenson:05rd} is the derivation of a lower bound for the difference of two discrete solutions in terms of the indicators on the refined area. The proof uses that the jump residual can be bounded by the element residual when the mesh-size is big with respect to $\kappa^{-1}$. The technique seems however not easily extendable to the higher order case. Because of the different structure of our estimator, the approach of \cite{Stevenson:05rd} does not carry over immediately. Nonetheless, we believe that having an error dominated oscillation is advantageous.

The rest of the paper is organized as follows: In \S \ref{S:prob_not} we introduce the notation  and recall the problem and some auxiliary results, in \S \ref{S:a_post} we derive the a posteriori error estimator and in \S \ref{S:conc} we draw the conclusion.

\section{Problem setting and auxiliary results}
\label{S:prob_not}


Assume $\Omega\subset\R^d$, $d\geq2$, is a 
bounded domain with Lipschitz polygonal boundary $\partial\Omega$, and assume $\omega\subseteq\Omega$ is 
a non-empty open measurable subset with Lipschitz boundary $\partial\omega$. 
We denote with $L^2(\omega)$ the set of square-integrable 
functions over $\omega$ and with $H^1(\omega)$ the set of functions in $L^2(\omega)$ with
distributional gradient also in $L^2(\omega)$. With $H^1_0(\omega)$ we indicate the subspace of $H^1(\omega)$  of functions with vanishing trace on $\partial\omega$. 
The dual space of $H^1_0(\omega)$ is denoted with $H^{-1}(\omega)$. 
Finally, $L^{\infty}(\omega)$ is the space of essentially bounded functions.
The $L^2$-norm is indicated with $\norm{\cdot}_\omega$ and the $L^\infty$-norm with $\norm{\cdot}_{\infty;\omega}$
 where we omit the subscript $\omega$, if 
$\omega=\Omega$. Finally $\dual{\cdot}{\cdot}$ denotes the duality pairing between $H^1_0(\omega)$ and 
  $H^{-1}(\omega)$. 

\subsection{Variational formulation of the problem}
 Given $f\in H^{-1}(\Omega)$ the standard variational formulation of \eqref{rd-intro} reads
\begin{equation}
 \label{prob:rd}
\textrm{find }u\in H^1_0(\Omega)\textrm{ such that } \forall v\in H^1_0(\Omega) 
    \quad \dual{\Op(u)}{v}=\dual{f}{v},
\end{equation}
where the linear operator $\Op: H^1_0(\Omega)\to H^{-1}(\Omega)$ is given by
\begin{equation*}
\dual{\Op(u)}{v}:=\int_\Omega \nabla u\cdot\nabla v+\kappa^2\int_\Omega uv,\qquad u,v\in H^1_0(\Omega).
\end{equation*}
We denote the energy norm on $\omega\subset\Omega$ and the corresponding dual norm with
\begin{equation*}
 \tnorm{\cdot}^2_\omega:=\norm{\nabla\cdot}^2_\omega+\kappa^2\norm{\cdot}_\omega^2,\qquad 
\tnorm{\cdot}_{*;\omega}:=\sup_{\vphi\in H^1_0(\omega)}\frac{\dual{\cdot}{\vphi}}{\tnorm{\vphi}_\omega}
\end{equation*}
and again omit the subscript $\omega$, if $\omega=\Omega$.  

The bilinear form $a:H^1_0(\Omega)\times H^1_0(\Omega)\to\R$ defined by $a(u,v):=\dual{\Op(u)}{v}$ satisfies $a(u,u)=\tnorm{u}^2$ and $a(u,v)\leq \tnorm{u}\tnorm{v}$ for every $u$,$v\in H^1_0(\Omega)$. Hence, the Lax-Milgram Lemma implies that problem \eqref{prob:rd} admits a unique solution $u\in H^1_0(\Omega)$ with $\tnorm{u}=\tnorm{f}_*$.

\subsection{Galerkin approximation}
We denote by $\tri$ a  conforming simplicial mesh of $\Omega$, by $\Afaces$ the set of 
its $(d-1)$-dimensional faces, and by $\verts$ the set of its vertices.  If $\el\in\tri$ is an element and $\fa\in\Afaces$ is a 
face, we write $|\el|$ and $|\fa|$ for its $d$-dimensional Lebesgue and 
$(d-1)$-dimensional Hausdorff measure, respectively. We assume that $\mesh$ belongs to a family of shape-regular meshes and satisfies
\[
\sup_{\el\in\mesh}\frac{h_\el}{\rho_\el}\leq\ShapePar<+\infty,
\] 
where $h_\el:=\diam(\el)$ and $\rho_\el$ is the 
maximum diameter of a ball inscribed in $\el$.

We recall the definition of the conforming finite element space
\[
 \FEspace^{1,0}(\tri)
 =
 \{v\in H^1(\Omega), \;
  \forall\, \el\in\mathcal{\tri},\, v\in\mathbb{P}_1(\el) \},
\]
and set $\FEspace:= \FEspace^{1,0}(\tri)\cap H^1_0(\Omega)$, 
which consists of all continuous functions that are piecewise affine over $\tri$ and have vanishing 
trace on $\partial\Omega$. We indicate with 
\[
\Op(\FEspace):=\{g\in H^{-1}, g=\Op(\vphi), \vphi\in \FEspace\}
\]
the set of functionals obtained by applying the operator $\Op$ to a discrete function. In particular note that
\begin{equation}
\label{def-s-tilde}
\Op(\FEspace)\subset\DiscrIm:=\Span\{\Ddelta_\fa, \fa\in\faces\}\oplus \Span\{p_\el\in\Poly_1(\el), \el\in\tri\},
\end{equation}
where $\Ddelta_\fa$ is a Dirac distribution on $\fa\in\faces$, i.e. $\dual{\Ddelta_\fa}{\vphi}=\int_\fa \vphi$, 
for every $\vphi\in H^1_0(\Omega)$.

The Galerkin solution $U\in \FEspace$ to problem \eqref{prob:rd} satisfies
\begin{equation}
 \label{discr_sol}
\int_{\Omega}\nabla U\cdot\nabla V+\kappa^2\int_{\Omega}U V=\dual{f}{V}, \qquad \forall \ V\in \FEspace.
\end{equation}
Existence and uniqueness of $U$ are guaranteed by the Lax-Milgram Lemma.

We denote the set of nodes of $\FEspace$ by $\nodes$, although they coincide with $\verts$. In view of an extension to higher order, it is in fact better to understand if the nodes come into play or the vertices. 
In the same spirit, we indicate with $\{\phi_z\}_{z\in\nodes}$ 
the nodal basis, for which each $\phi_z$ is uniquely 
determined by 
\[
\phi_z\in \FEspace \quad\textrm{and}\quad \forall y\in\nodes\quad \phi_z(y)=\delta_{zy},
\]
and with $\{\lambda_z\}_{z\in\verts}$ 
the nodal basis of $S^{1,0}(\tri)$. We also recall that $\{\lambda_z\}_{z\in\verts}$ form a partition of unity, i.e.\ $\sum_{z\in\verts}\lambda_z=1$.  
Here and in the following, a subscript $\fa$, $\el$,  
$\Omega$, etc.\ to $\verts$, $\nodes$, $\Afaces$ indicates that only those vertices, nodes or 
faces that are contained in the index set are considered. With any vertex $z\in\verts$ we associate the sets
\[
\omega_z:=\supp (\lambda_z)\qquad\textrm{and}
\qquad \sigma_z:=\bigcup\{\fa\in\Afaces, \fa\ni z\},
\]
i.e.\ the star around $z$ and its skeleton. For every face $\fa\in \Afaces$, the set $\omega_{\fa} := \cup\{\el\in\tri: \partial\el\supseteq\fa \}$
is the union of elements sharing the face $\fa$.  We write simply $\el\subset\omega_\fa$ or $\el\subset\omega_z$
 but intend $\el\in\tri$ and $\el\subset\omega_\fa$ or $\el\subset\omega_z$. Moreover 
$h_\fa:=\max\{h_\el, \el\subset\omega_\fa\}$ is the bigger of the diameters
of the elements in $\omega_\fa$.

\subsection{Scaling properties} 
Here and in what follows we write $A\lesssim B$ and $A\gtrsim B$ with the meaning that there exist constants 
$C$, $c>0$, such that $A\leq C B$ and $A\geq c B$ respectively, where $C$, $c$ may depend on $\ShapePar$ and $d$,
 but are 
independent of other parameters like $h_\el$ and in particular of $\kappa$. We write also $A\approx B$ when both
$A\gtrsim B$ and $A\lesssim B$ hold.

In order to bound the norms of basis and dual basis functions, the standard procedure is to transform to the reference element $\widehat{\el}:=\mathrm{conv hull}\{{\bf{0}}, e_1,\ldots, e_d\}$, see \cite[p.\ 117-120]{Ciarlet:78}.  For the $L^\infty$-norm we obtain,
\begin{equation}
\label{scaling-inf}
\norm{\lambda_z}_{\infty;\omega_z}\leq 1,\qquad \norm{\nabla\lambda_z}_{\infty;\omega_z}\lesssim h_z^{-1},\qquad \forall\ z\in\verts,
\end{equation}
while for the $L^2$-norm we have
\begin{equation}
\label{scaling-L2}
\norm{\phi_z}_{\el}\approx |\el|^{1/2}, \qquad \norm{\nabla\phi_z}_{\el}\lesssim h_\el^{-1}|\el|^{1/2}, \qquad z\in\nodes.
\end{equation}

\subsection{Poincar\'e-Friedrichs and trace inequalities}
 For every $\vphi\in H^1(\omega)$ the Poincar\'e inequality says
\begin{equation}
 \label{poinc}
\norm{\vphi-\frac{1}{|\omega|}\int_\omega \vphi}_\omega\lesssim \diam(\omega)\norm{\nabla\vphi}_\omega.
\end{equation}
In case $\vphi$ vanishes over a part of the boundary $\Gamma\subset\partial\omega$, with $|\Gamma|\neq0$, 
the Friedrichs inequality holds
\begin{equation}
 \label{fried}
\norm{\vphi}_\omega\lesssim \diam(\omega)\norm{\nabla\vphi}_\omega.
\end{equation}
The hidden constants in \eqref{poinc}--\eqref{fried} have been explicitly bounded in \cite{Veeser.Verfuerth:12} 
in case $\omega$ is a finite element star.
In addition, given $\fa\subset\faces$ and $\el\subset\omega_\fa$, we also use the trace inequality 
\begin{equation}
\label{trace}
 \norm{\vphi}_\fa\lesssim\norm{\vphi}_\el^{1/2}\norm{\nabla\vphi}_\el^{1/2}, 
\end{equation}
where $\vphi$ vanishes on a face $\fa'$ on the boundary of $\el$.

\section{Robust error dominated oscillation}
\label{S:a_post}
\subsection{The residual and its localization}
The first step towards a posteriori error estimation is the equivalence between energy norm of the error 
and dual norm of the residual, where the  residual is defined as
\[
\Res:=f-\Op(U)\in H^{-1}(\Omega).
\]
Since we consider the energy norm, we even have the equality 
\begin{equation}
\label{err=res}
\tnorm{u-U}=\tnorm{\Res}_{*}.
\end{equation}
Moreover, as a consequence of the 
Galerkin orthogonality, the residual is orthogonal to the discrete space, i.e.\
\[
\dual{\Res}{\vphi}=0, \quad\text{ for every }\quad \vphi\in\FEspace.
\]
As proved in, e.g., \cite{Kreuzer.Veeser:16} for the $H^{-1}$-norm, thanks to this property it is possible to 
localize the dual norm of the residual on stars. For completeness, 
we provide the proof for the $\tnorm{\cdot}_*$-norm.

\begin{prop}[Localization]
\label{P:loc}
 For every $g\in H^{-1}(\Omega)$ it holds 
\begin{equation}
\label{sum<int}
\sum_{z\in\verts}\tnorm{g}_{*;\omega_z}^2\lesssim\tnorm{g}^2_{*}. 
\end{equation}
If additionally $g$ is such that 
\begin{equation}
\label{hp-ort}
 \dual{g}{\vphi}=0 \quad\text{ for every }\quad \vphi\in \FEspace,
\end{equation}
 then
\begin{equation}
\label{int<sum}
\tnorm{g}^2_{*}\lesssim \sum_{z\in\verts}\tnorm{g}_{*;\omega_z}^2.
\end{equation}
\end{prop}
\begin{proof}
In order to prove \eqref{sum<int}, take for every 
$z\in\verts$ a function $\vphi_z\in H^1_0(\omega_z)$ with $\tnorm{\vphi_z}_{\omega_z}\leq 1$  and set 
$\vphi=\sum_{z\in\verts}\dual{g}{\vphi_z}\vphi_z\in H^1_0(\Omega)$.  We have
\begin{equation*}
 \sum_{z\in\verts}\dual{g}{\vphi_z}^2=\dual{g}{\vphi}\leq \tnorm{g}_{*}\tnorm{\vphi}
\end{equation*}
and
\begin{equation*}
 \tnorm{\vphi}^2=\sum_{\el\in\tri}\tnorm{\vphi}^2_\el
\lesssim \sum_{\el\in\tri}\sum_{z\in\verts_\el}\dual{g}{\vphi_z}^2\tnorm{\vphi_z}^2_\el
\lesssim \sum_{z\in\verts}\dual{g}{\vphi_z}^2.
\end{equation*}
We obtain \eqref{sum<int} taking the suprema over all $\vphi_z$ for every $z\in\verts$.
\\
Turning to \eqref{int<sum}, we exploit that $\{\lambda_z\}_{z\in\verts}$ form a partition of unity and \eqref{hp-ort} 
to obtain, for every $\vphi\in H^1_0(\Omega)$,
\begin{equation}
\begin{aligned}
\dual{g}{\vphi}&=\sum_{z\in\verts}\dual{g}{\vphi\lambda_z}=\sum_{z\in\verts}\dual{g}{(\vphi-c_z)\lambda_z}\\
&
\leq \left(\sum_{z\in\verts}\tnorm{g}^2_{*;\omega_z}\right)^{1/2} \left(\sum_{z\in\verts}
\tnorm{(\vphi-c_z)\lambda_z}^2_{\omega_z}\right)^{1/2}.
\end{aligned}
\label{dual-part-un-ort}
\end{equation}
for every $c_z\in\R$ such that $c_z=0$ if $z\in\verts_{\partial\Omega}$. The choice 
$c_z=1/|\omega_z|\int_{\omega_z}\vphi$ if $z\in\verts_\Omega$, not only allows the
 use of a Poincar\'e inequality, but also implies that $\norm{\vphi-c_z}_{\omega_z}\leq\norm{\vphi}_{\omega_z}$.
Together with the Friedrichs inequality \eqref{fried} on the stars with $z\in\verts_{\partial\Omega}$ and the scaling properties
\eqref{scaling-inf}, we get
\begin{equation*}
\begin{aligned}
 \tnorm{(\vphi-c_z)\lambda_z}^2_{\omega_z}&\leq\kappa^2\norm{(\vphi-c_z)\lambda_z}^2_{\omega_z}
	+2\norm{\lambda_z\nabla \vphi}^2_{\omega_z}+2\norm{(\vphi-c_z)\nabla \lambda_z}^2_{\omega_z} 
\\
&\leq
\kappa^2\norm{\vphi}^2_{\omega_z}+2\norm{\nabla \vphi}^2_{\omega_z}+2\norm{\vphi-c_z}^2_{\omega_z}\norm{\nabla\lambda_z}_{\infty;\omega_z}
\\
&\lesssim \tnorm{\vphi}^2_{\omega_z}.
\end{aligned}
\end{equation*}
Inserting into \eqref{dual-part-un-ort}, dividing by $\tnorm{\vphi}$ 
and taking the supremum over $\vphi\in H^1_0(\Omega)$ gives the assertion.
\end{proof}

Applying Proposition \ref{P:loc} to the residual, we arrive at
\begin{equation}
\label{err-appx-sum-I}
\tnorm{u-U}^2\approx\sum_{z\in\verts}\tnorm{f-\Op (U)}^2_{*,\omega_z}.
\end{equation}
Moreover each quantity $\tnorm{f-\Op (U)}_{*,\omega_z}$ also yields a lower bound for the local error: 
\begin{equation}
\label{local-lower-f}
\begin{aligned}
\tnorm{f-\Op (U)}_{*;\omega_z}&=\sup_{\vphi\in H^1_0(\omega_z)}\frac{\dual{f-\Op (U)}{\vphi}}{\tnorm{\vphi}_{\omega_z}}
\\
&=\sup_{\vphi\in H^1_0(\omega_z)}\frac{\int_{\omega_z}\nabla(u-U)\cdot\nabla\vphi+\kappa^2(u-U)\vphi}{\tnorm{\vphi}_{\omega_z}}
\\
&
\leq \tnorm{u-U}_{\omega_z}. 
\end{aligned}
\end{equation}

\subsection{Computability and oscillation}
The quantities $\tnorm{f-\Op (U)}_{*;\omega_z}$ provide a two-sided bound for the error. Yet, they are not suited for a posteriori error estimation. In fact, without a priori knowledge on $f$, they cannot be bounded by testing $f$ with a finite number of test functions, see \cite[Cor.\ 5]{Kreuzer.Veeser:16}.

Hence, we insert a computable approximation $\interp f$ to $f$ from a finite dimensional space and write
\begin{equation}
\label{ins-op}
\tnorm{f-\Op (U)}_{*,\omega_z}\leq \tnorm{\interp f-\Op (U)}_{*,\omega_z}+\tnorm{f-\interp f}_{*,\omega_z}.
\end{equation}
It would be desirable that the right-hand side in \eqref{ins-op} is also a lower bound for the left-hand side. Necessary and sufficient conditions on the operator $\interp$ for this to hold have been derived by Kreuzer and Veeser \cite{Kreuzer.Veeser:16}. We recall the proof for completeness.

\begin{prop}
\label{P:nec-suf-cond}
Assume $f\in H^{-1}(\Omega)$, $U\in S$ and $z\in\verts$. Then
\begin{equation}
\label{inv-ineq}
\tnorm{f-\Op (U)}_{*,\omega_z}\gtrsim \tnorm{\interp f-\Op (U)}_{*,\omega_z}+\tnorm{f-\interp f}_{*,\omega_z}
\end{equation}
if and only if $\interp$ satisfies
\begin{subequations}
\label{prop-op}
\begin{align}
\label{inva-op} \interp \Op (V)&=\Op (V) &&\!\forall \ V\in S\phantom{g\in H^{-1}(\Omega)} &&\textrm{(invariance)}\\
\label{stab-op} \tnorm{\interp g}_{*,\omega_z}&\lesssim\tnorm{g}_{*,\omega_z} &&\forall \ g\in H^{-1}(\Omega)\phantom{V\in S} &&\textrm{(local stability).}
\end{align}
\end{subequations}
\end{prop}
\begin{proof}
Assume first \eqref{inv-ineq}. Choosing respectively $f=\Op (U)$ and $U=0$ gives  \eqref{inva-op} and \eqref{stab-op}. 
\\
On the other hand, assume the validity of \eqref{prop-op}. Then, the triangle inequality yields
\[
\tnorm{f-\interp f}_{*,\omega_z}\leq \tnorm{f-\Op (U)}_{*,\omega_z}+\tnorm{\interp f-\Op (U)}_{*,\omega_z}
\]
and 
\[
\tnorm{\interp f-\Op (U)}_{*,\omega_z}=\tnorm{\interp (f-\Op (U))}_{*,\omega_z}\lesssim \tnorm{f-\Op (U)}_{*,\omega_z},
\]
which prove \eqref{inv-ineq}.
\end{proof}

\subsection{Abstract construction of the interpolation operator}\label{abstract-construction}

Still following the lines in \cite{Kreuzer.Veeser:16} we first construct abstractly an operator satisfying \eqref{prop-op} via a locally stable bi-orthogonal system. To this end we first set $\phi_{z;\el}:=\phi_z\chi_\el$ for all $\el \in \tri$ and $z \in \nodes_\el$ and recall that $\{ \phi_{z;\el}, \Ddelta_\fa \}_{\el\in\tri, z\in\nodes_\el,\fa\in\faces}$ forms a basis of $\DiscrIm$. Moreover assume $\{\phi^*_{z;\el}, \phi^*_\fa \}_{\el\in\tri, z\in\nodes_\el,\fa\in\faces} \subset H^1_0(\Omega)$ and define $\Pi : H^{-1}(\Omega) \to \DiscrIm$ for every $g \in H^{-1}(\Omega)$ by
\begin{equation}\label{abstract-op}
\Pi g := \sum_{\el \in \tri} \sum_{z \in \nodes_\el} \langle g, \phi^*_{z;\el} \rangle \phi_{z;\el} + \sum_{\fa \in \faces} \langle g, \phi^*_\fa \rangle \Ddelta_\fa.
\end{equation}
We first give conditions on $\{\phi^*_{z;\el}, \phi^*_\fa \}$ so that $\Pi$ satisfies \eqref{prop-op} and in the next section give a suitable example of such functions.

\begin{lem}
\label{L:inv-stab-abstract}
If $\{ \phi_{z;\el}, \Ddelta_\fa \}$ and $\{\phi^*_{z;\el}, \phi^*_\fa \}$, $\el\in\tri$,  $z\in\nodes_\el$, $\fa\in\faces$ form a bi-orthogonal system, that is
\begin{subequations}
\label{bi-orth-inv}
\begin{align}
\label{bi-orth-inv-a} \langle \phi_{z_1;\el_1}, \phi^*_{z_2;\el_2} \rangle &= \delta_{z_1z_2} \delta_{\el_1\el_2} &&\forall \el_j \in \tri, z_j \in \nodes_{\el_j}, j \in \{1, 2\} \\
\label{bi-orth-inv-b} \langle \Ddelta_{\fa_1}, \phi^*_{\fa_2} \rangle &= \delta_{\fa_1\fa_2} &&\forall \fa_1, \fa_2 \in \faces \\
\label{bi-orth-inv-c} \langle \phi_{z;\el}, \phi^*_\fa \rangle &= \langle \Ddelta_\fa, \phi^*_{z;\el} \rangle = 0 &&\forall \el \in \tri, z \in \nodes_\el, \fa \in \faces
\end{align}
\end{subequations}
then $\Pi$ defined in \eqref{abstract-op} is invariant over $\DiscrIm$ satisfying \eqref{inva-op}.
\\
Moreover, if for every $y \in \verts$
\begin{subequations}\label{bi-orth-stab}
\begin{align}
\label{bi-orth-stab-a} \tnorm{\phi_{z;\el}}_{*;\omega_y} \tnorm{\phi^*_{z;\el}}_{\omega_y} &\lesssim 1 &&\forall \el \subset \omega_y \\
\label{bi-orth-stab-b} \tnorm{\Ddelta_\fa}_{*;\omega_y} \tnorm{\phi^*_\fa}_{\omega_y} &\lesssim 1 &&\forall \fa \subset \Afaces_{\omega_y}
\end{align}
\end{subequations}
then $\Pi$ is locally stable in the $\tnorm{\cdot}_{*}$-norm satisfying \eqref{stab-op}.
\end{lem}

\begin{proof}
Invariance \eqref{inva-op} follows readily recalling that $\{ \phi_{z;\el}, \Ddelta_\fa \}_{\el\in\tri, z\in\nodes_\el,\fa\in\faces}$ is a basis of $\DiscrIm$ and applying \eqref{bi-orth-inv}.
\\
Concerning \eqref{stab-op} for every $\varphi \in H^1_0(\omega_y)$ we use \eqref{bi-orth-stab} to get
\begin{equation*}
\begin{split}
\langle \Pi g, \varphi \rangle &=  \sum_{\el \subset \omega_y} \sum_{z \in \nodes_\el} \langle g, \phi^*_{z;\el} \rangle \langle \phi_{z;\el}, \varphi \rangle + \sum_{\fa \in \Afaces_{\omega_y}} \langle g, \phi^*_\fa \rangle \langle \Ddelta_\fa, \varphi \rangle \\
&\le \sum_{\el \subset \omega_y} \sum_{z \in \nodes_\el} \tnorm{g}_{*;\omega_y} \tnorm{\phi^*_{z;\el}}_{\omega_y} \tnorm{\phi_{z;\el}}_{*;\omega_y} \tnorm{\varphi}_{\omega_y} \\
&\quad+ \sum_{\fa \in \Afaces_{\omega_y}} \tnorm{g}_{*;\omega_y} \tnorm{\phi^*_\fa}_{\omega_y} \tnorm{\Ddelta_\fa}_{*;\omega_y} \tnorm{\varphi}_{\omega_y} \\
&\lesssim \tnorm{g}_{*;\omega_y} \tnorm{\varphi}_{\omega_y}.
\end{split}
\end{equation*}
Dividing by $\tnorm{\varphi}_{\omega_y}$ and taking the supremum over $\varphi \in H^1_0(\omega_y)$ we get \eqref{stab-op}.
\end{proof}

In \cite{Kreuzer.Veeser:16} two possible bi-orthogonal systems for the construction of $\Pi$ have been given for the case $\Op:=-\Delta$ and the $\tnorm{\cdot}_*$-norm replaced by the $H^{-1}$-norm. Since they are not robustly stable in the $\tnorm{\cdot}_*$-norm, they are not suited for our purposes. We therefore modify the construction of one of them with the help of the squeezed bubble functions of Verf\"urth \cite{Verfuerth:98}.

\subsection{Bubble functions}

\label{SS:bubble}
We start with the construction of $\{ \phi^*_{z;\el} \}_{\el\in\tri, z\in\nodes_\el}$ and recollect from \cite[\S 3.2.3, \S 3.6]{Verfuerth:13} the definition of the element bubble functions, together with their properties. For every $\el\in\tri$ we set
\begin{equation*}
\bub_\el:=\prod_{z\in\verts_\el}\lambda_z
\end{equation*}
and recall that, for 
every $v\in\Poly_1(\el)$, there holds
\begin{subequations}
 \begin{align}
\label{eb-std-L2}
 \norm{v}_\el&\lesssim \norm{\bub_\el^{1/2}v}_\el, \qquad\textrm{and}\\
\label{eb-std-H1}
\norm{\nabla(\bub_\el v)}_\el&\lesssim h_\el^{-1}\norm{v}_\el.
\end{align}
  \end{subequations}
Note that $\bub_\el\in H^1_0(\el)$, and that in \cite{Kreuzer.Veeser:16} scaled variants of  $\bub_\el$, $\el \in \tri$ are part of the dual basis of the bi-orthogonal system upon which the operator in \S \ref{aposterr} is defined. Yet, this leads to an operator that is invariant only onto piecewise constant functions. In view of \eqref{bi-orth-inv-a} and $\phi_{z;\el} \in \Poly_1(\el)$ we must employ other bubble functions in order to obtain orthogonality and therefore invariance onto polynomials of degree $1$.

 To this end, for every $\el\in\tri$ and $z\in\nodes_\el$,
 we denote with $\phi_{z;\el}:=\phi_z\chi_\el$ the restriction of $\phi_z$ to $\el$. Moreover let $\{\psi_z^\el\}_{z\in\nodes_\el}$ be the local dual basis to $\{\phi_{z;\el}\}_{z\in\nodes_\el}$ with respect to the $L^2$-scalar product with weight $\bub_\el$, that is, 
\begin{equation*}
\psi_z^\el\in \Poly_1(\el) \text{ and for every }y\in\nodes_\el, \quad \int_\el \bub_\el \psi_z^\el\phi_{y;\el}=\delta_{yz}.
\end{equation*}
For every $z\in\nodes_\el$ we then set  
\begin{equation*}
\phi^*_{z;\el}:= \psi_z^\el\bub_\el,
\end{equation*}
so that $\phi^*_{z;\el}\in H^1_0(\el)\cap \Poly_{2+d}(\el)$ and for every $y\in\nodes$ we have the orthogonality property
\begin{equation}
\label{dual-basis}
\int_\el \phi^*_{z;\el}\phi_y=\delta_{zy}.
\end{equation}
Concerning the scaling of norms of $\phi^*_{z;\el}$, we observe that on the reference element, the dual basis of $\{\widehat{\phi}_{\widehat{z}}\}_{\widehat{z}\in\nodes_{\widehat{\el}}}:=\{\phi_z\circ G_\el\}_{z\in\nodes_\el}$ is given by $\widehat{\psi}_{\widehat{z}}=\det(\mathbf{G}_\el)\cdot\psi_z^\el\circ G_\el$, where $\mathbf{G}_\el$ is the non singular matrix associated with one of the affine transformations $G_\el:\R^d\to\R^d$ such that $G_\el(\widehat{\el})=\el$. Exploiting also an inverse inequality, we arrive at 
\begin{equation*}
\norm{\nabla \phi^*_{z;\el}}_\el\lesssim h_\el^{-1}\norm{\phi^*_{z;\el}}_\el\leq h_{\el}^{-1}\norm{\psi^\el_z}_{\el}\lesssim h_{\el}^{-1}|\el|^{-1/2},
\end{equation*}
and 
\begin{equation}
\label{scaling-dual-k}
\tnorm{\phi^*_{z;\el}}_\el\lesssim |\el|^{-1/2}\max\{h_\el^{-1},\kappa\}.
\end{equation}

For the construction of the functions $\phi^*_\fa$ associated with the interior faces we use the squeezed bubble functions 
introduced in \cite{Verfuerth:98}.  Given a parameter $\vartheta\in(0,1]$, we consider the transformation $\Phi_\vartheta: \R^d\to\R^d$ that maps 
$(x_1, \ldots, x_d)$ to $(x_1,\ldots, x_{d-1}, \vartheta x_d)$. Given a face $\fa$ and an element $\el$ with 
$\fa\subset\partial\el$ we denote by $\transf_{\el,\fa}:\R^d\to\R^d$ an orientation-preserving
 affine transformation which maps the reference element $\widehat{\el}$ to the element $\el$ and the face 
$\widehat{\fa}:=\{x\in \widehat{\el}, x_d=0 \}$ to $\fa$. We define the `squeezed' element $\el_\vartheta$ by
\[
\el_\vartheta:=\transf_{\el,\fa}\circ\Phi_\vartheta\circ \transf_{\el,\fa}^{-1}(\el)
\] 
and set $\omega_\fa^\vartheta:=\cup_{\el\subset\omega_\fa}\el_\vartheta$. Moreover we
denote by $\lambda_{z,\vartheta}$, $z\in\verts_{\fa}$, the piecewise affine and continuous functions such that $\lambda_{z,\vartheta}(y)=\delta_{yz}$ for every $y\in\verts_{\omega^\vartheta_\fa}$, and extend them by zero outside $\omega^\vartheta_\fa$.
We set 
\begin{equation}
\label{theta}
\vartheta:=\min\{1,h_\el^{-1}\kappa^{-1}\}
\end{equation}
and define the `squeezed' bubble function $\bub_\fa$ as
\begin{equation}
\label{def-face-bubble-low-b}
 \bub_{\fa}:=\prod_{z\in\verts_\fa}\lambda_{z,\vartheta}.
\end{equation}
There holds $\bub_{\fa}\in H^1_0(\omega_\fa)$ and, for every $v\in\Poly_\ell(\fa)$, see \cite[\S 3.6]{Verfuerth:13}
\begin{subequations}
\begin{equation}
\label{fb-std-L2}
 \norm{v}_\fa\lesssim \norm{\bub_\fa^{1/2}v}_\fa, \qquad\textrm{and}
\end{equation}
\begin{equation}
\label{fb-std-H1} \norm{\bub_\fa v}_\el\lesssim \vartheta^{1/2}h_\fa^{1/2}\norm{v}_\fa,\qquad \norm{\nabla(\bub_\fa v)}_\el\lesssim \vartheta^{-1/2}h_\fa^{-1/2}\norm{v}_\fa.
\end{equation}
  \end{subequations}
The scaled bubble function
\begin{equation}
\label{def-face-bubble}
 \psi_{\fa}:=\left(\int_\fa \bub_\fa\right)^{-1}\bub_\fa
\end{equation}
thus satisfies 
\begin{subequations}
\begin{equation}
\label{fb:supp-psi}
\psi_{\fa}\in H^1_0(\omega_\fa),\qquad \dual{\Ddelta_\fa}{\psi_{\fa}}=1,\qquad \text{ and }
\end{equation}
\begin{equation}
\label{fb:stab}
\tnorm{\psi_{\fa}}_{\omega_\fa}\lesssim |\fa|^{-1/2}\max\{h_\fa^{-1}, \kappa\}^{1/2}.
\end{equation}
\end{subequations}

\subsection{An $\Op(S)$-invariant and locally stable interpolation operator} 
In this section we construct an interpolation operator that satisfies \eqref{prop-op} with the help of the bubble functions of \S \ref{SS:bubble}. We define $\interp:H^{-1}(\Omega)\to \DiscrIm$ as in \eqref{abstract-op} with
\begin{equation*}
\begin{aligned}
\phi^*_{z;\el} &:= \psi^\el_z \bub_\el \in H^1_0(\el) &&\forall \el \in \tri, z \in \nodes_\el, \\
\phi^*_\fa &:= \psi_\fa - \sum_{\el \subset \omega_\fa} \sum_{z \in \nodes_\el} \phi^*_{z;\el}  \langle \phi_{z;\el}, \psi_\fa \rangle \in H^1_0(\omega_\fa) &&\forall \fa \in \faces.
\end{aligned}
\end{equation*}
In order to be able to apply the operator $\interp$ to $f$, we require that quantities of the form
$\dual{f}{\Phi}$ are known exactly, where $\Phi\in H^1_0(\Omega)$ is a product of polynomials $p_\el\in\Poly_1(\el)$ with basis functions $\lambda_z$ or $\lambda_{z,\vartheta}$, $z\in\verts_\Omega$. Note that this is slightly more demanding than requiring that the terms $\dual{f}{\phi_z}$ are available, and this information is needed in order to build the linear system that has to be solved  to find the discrete solution $U$. 

The following lemma proves that $\interp$ satisfies \eqref{prop-op}.

\begin{lem}\label{L:inv-stab-concret}
The functions $\{ \phi^*_{z;\el}, \phi^*_\fa \}_{\el \in \tri, z \in \nodes_\el, \fa \in \faces}$ satisfy the assumptions \eqref{bi-orth-inv} and \eqref{bi-orth-stab} of Lemma \ref{L:inv-stab-abstract}. Therefore the operator $\Pi$ of \eqref{abstract-op} satisfies \eqref{prop-op}.
\end{lem}

\begin{proof}
Conditions \eqref{bi-orth-inv-a} and \eqref{bi-orth-inv-b} follow readily from \eqref{dual-basis}, \eqref{fb:supp-psi} and $\phi^*_{z;\el} \in H^1_0(\el)$.
\\
Moreover, for every $\widetilde \el \subset \omega_\fa$ and $y \in \nodes_{\widetilde \el}$ we get from \eqref{bi-orth-inv-a}
\begin{equation*}
\begin{split}
\langle \phi_{y;\widetilde \el}, \phi^*_\fa \rangle &= \langle \phi_{y;\widetilde \el}, \psi_\fa \rangle - \sum_{\el \subset \omega_\fa} \sum_{z \in \nodes\el} \langle \phi_{y;\widetilde \el}, \phi^*_{z;\el} \rangle \langle \phi_{z;\el}, \psi_\fa \rangle \\
&= \langle \phi_{y;\widetilde \el}, \psi_\fa \rangle - \langle \phi_{y;\widetilde \el}, \psi_\fa \rangle \\
&= 0.
\end{split}
\end{equation*}
Since $\phi^*_\fa \in H^1_0(\omega_\fa)$ and $\phi^*_{z;\el} \in H^1_0(\el)$ this proves \eqref{bi-orth-inv-c}.
\\
To prove \eqref{bi-orth-stab-a} we must in particular bound $\tnorm{\phi_{z;\el}}_{*;\omega_y}$. To this end we observe that for every $\varphi \in H^1_0(\omega_y)$ and for every $\el \subset \omega_y$ and $z \in \nodes_\el$ Friedrichs inequality \eqref{fried} and the scaling property \eqref{scaling-L2} imply
\begin{equation}
\label{intKphiz_vphi}
\int_\el\phi_{z;\el}\vphi\leq\norm{\phi_{z;\el}}_\el\norm{\vphi}_\el
\lesssim |\el|^{1/2}\min\{\kappa^{-1}, h_\el\}\tnorm{\vphi}_{\el}.
\end{equation}
Dividing by $\tnorm{\vphi}_{\el}$, taking the supremum over $\varphi \in H^1_0(\omega_y)$ and combining with \eqref{scaling-dual-k} gives \eqref{bi-orth-stab-a}.
\\
Finally, we prove \eqref{bi-orth-stab-b}. Thanks to the trace inequality \eqref{trace} and the Young inequality 
$ab\leq \frac{\veps}{2} a^2+\frac{1}{2\veps} b^2 $ with $\veps=\kappa^{1/2}$, we have, for every $\vphi\in H^1_0(\omega_y)$, 
\begin{equation}
\label{intFvphi}
\begin{split}
\int_\fa\vphi&\leq|\fa|^{1/2}\norm{\vphi}_\fa\lesssim|\fa|^{1/2}\norm{\vphi}^{1/2}_\el\norm{\nabla\vphi}^{1/2}_\el
\\
&\lesssim |\fa|^{1/2}\min\{h_\el, \kappa^{-1}\}^{1/2}\tnorm{\vphi}_\el,
\end{split}
\end{equation}
where $\el\subset\omega_\fa$. Therefore, we get
\begin{equation*}
\tnorm{\Ddelta_\fa}_{*;\omega_\fa} \lesssim |\fa|^{1/2}\min\{\kappa^{-1}, h_\el\}^{1/2}.
\end{equation*}
On the other hand, exploiting \eqref{bi-orth-stab-a}, we obtain
\begin{equation*}
\tnorm{\phi^*_\fa}_{\omega_\fa} \lesssim \tnorm{\psi_\fa}_{\omega_\fa} \Bigl( 1 + \sum_{\el \subset \omega_\fa} \sum_{z \in \nodes_\el} \tnorm{\phi^*_{z;\el}}_{\omega_\fa} \tnorm{\phi_{z;\el}}_{*;\omega_\fa} \Bigr) 
\lesssim \tnorm{\psi_\fa}_{\omega_\fa}
\end{equation*}
which gives \eqref{bi-orth-stab-b}.
\end{proof}

\subsection{A posteriori error estimator}\label{aposterr}

For every $z\in\verts$, we define the error estimators  
\begin{equation}
\Estim(U,z):=\tnorm{\interp f -\Op (U)}_{*;\omega_z}
\label{def_indic}
\end{equation}
with $\interp$ as in the previous section.  The following theorem states that the indicators together with the related oscillation provide a two-sided robust bound for the error of the Galerkin solution of \eqref{discr_sol}. 
\begin{thm}[Upper and Lower Bound]
The following global upper bound for the error holds true
\[
\tnorm{u-U}\lesssim \left(\sum_{z\in\verts}
\Estim(U,z)^2+\tnorm{f-\interp f}^2_{*;\omega_z}
\right)^{1/2}
\] 
Moreover the local lower bounds also hold
\[
\Estim(U,z)+\tnorm{f-\interp f}_{*;\omega_z}\lesssim \tnorm{u-U}_{\omega_z}, \qquad\forall z\in\verts.
\]
The hidden constants are independent of $\kappa$.
\end{thm}
\begin{proof}
Combine Lemmas \ref{L:inv-stab-abstract}, \ref{L:inv-stab-concret} with Proposition \ref{P:nec-suf-cond} and \eqref{err-appx-sum-I}, \eqref{local-lower-f}.
\end{proof}

The indicators defined in \eqref{def_indic} can also be rewritten as $L^2$-norms  of polynomials on elements and on faces, thus mimicking
the standard structure of residual error indicators, and showing that they are computable, provided $\interp f$ is available. 
 For every $\el\in\tri$ and every $\fa\in\faces$ we set
\begin{equation*}
 \begin{aligned}
  r|_\el&:=\sum_{y\in\nodes_\el}\dual{f}{\phi^*_{y;\el}}\phi_{y;\el}-\kappa^2 U\\
  j|_\fa&:=\dual{f}{\phi^*_{\fa}}-[\nabla U\cdot n_\fa]_\fa
 \end{aligned}
\end{equation*}
 where $n_\fa$ is a unit vector normal to $\fa$, and $[\cdot]_\fa$ denotes the jump across $\fa$ in the direction of $n_\fa$. 
  
\begin{prop}
The error indicator defined in \eqref{def_indic} satisfies
\begin{equation*}
\begin{aligned}
 \Estim(U,z)&\approx \left(\sum_{\el\subset\omega_z}\min\{h_\el,\kappa^{-1}\}^2\norm{r}_\el^2\right)^{1/2}
+\left(\sum_{\fa\in\sigma_z}\min\{h_\fa,\kappa^{-1}\}\norm{j}^2_\fa\right)^{1/2}.
\end{aligned}
\end{equation*}
\end{prop}

\begin{proof}
We integrate by parts elementwise and derive the following $L^2$-re\-pre\-sen\-ta\-tion,
where the element and the face contributions are separated. For every $\vphi\in H^1_0(\omega_z)$ we write
\begin{equation*}
\begin{aligned}
&\dual{\interp f-\Op(U)}{\vphi}=
  \dual{\interp f}{\vphi}
      -\int_{\omega_z}\nabla U\nabla \vphi-\kappa^2\int_{\omega_z}U\vphi
\\
&\qquad=
\sum_{\el\subset\omega_z}\int_\el\left(\sum_{y\in\nodes_\el}\dual{f}{\phi^*_{y;\el}}\phi_{y;\el} -\kappa^2 U\right)\vphi
\\
&\qquad\qquad+\sum_{\fa\in\sigma_z}\int_\fa\left(\dual{f}{\phi^*_\fa}-[\nabla U\cdot n_\fa]_\fa\right) \vphi
\\
&\qquad=\sum_{\el\subset\omega_z}\int_\el r\vphi+\sum_{\fa\in\sigma_z}\int_\fa j\vphi.
\end{aligned}
\end{equation*}
Concerning the $\lesssim$-direction, we recall \eqref{intKphiz_vphi} and \eqref{intFvphi} and get
\begin{equation*}
\begin{aligned}
\dual{\interp f-\Op(U)}{\vphi}&= \sum_{\el\subset\omega_z}\int_\el r\vphi+\sum_{\fa\in\sigma_z}\int_\fa j\vphi
  \\
 &\lesssim \left(\sum_{\el\subset\omega_z}\min\{h_\el,\kappa^{-1}\}^2\norm{r}_\el^2\right)^{1/2}\tnorm{\vphi}_{\omega_z}
\\
&\qquad
+\left(\sum_{\fa\in\sigma_z}\min\{h_\fa,\kappa^{-1}\}\norm{j}^2_\fa\right)^{1/2}\tnorm{\vphi}_{\omega_z}.
\end{aligned}
\end{equation*}
Dividing by $\tnorm{\vphi}_{\omega_z}$ and taking the supremum over $\vphi\in H^1_0(\omega_z)$ proves 
 the first inequality.

Concerning the $\gtrsim$-direction, we start with proving
\begin{equation}
\label{ind-gtrsim-el-res}
\Estim(U,z)\gtrsim \left(\sum_{\el\subset\omega_z}\min\{h_\el,\kappa^{-1}\}^2\norm{r}_\el^2\right)^{1/2}.
\end{equation}
To this end, for every $\el\in\tri$, choose $\vphi= r\bub_\el$. Thanks to \eqref{eb-std-L2} 
we arrive at
\begin{equation}
\label{lb:el-1}
 \begin{aligned}
&\dual{\interp f-\Op(U)}{\vphi}
    =\int_\el r\vphi=\int_\el r^2\bub_\el\gtrsim 
\norm{r}^2_{\el}.
\end{aligned}
\end{equation}
Moreover $\norm{\bub_\el}_{\infty;\el}\leq1$ and \eqref{eb-std-H1} imply
\begin{equation*}
 \begin{aligned}
  \tnorm{\vphi}^2_{\el}=\norm{\nabla(r\bub_\el)}^2_\el+\kappa^2\norm{r\bub_\el}^2_\el
\lesssim \max\{h_\el^{-1},\kappa\}^2\norm{r}^2_\el.
 \end{aligned}
\end{equation*}
Combining this with \eqref{lb:el-1} and summing over $\el\subset\omega_z$ gives \eqref{ind-gtrsim-el-res}.
In order to prove 
\begin{equation}
\label{ind-gtrsim-fa-res}
\Estim(U,z)\gtrsim \left(\sum_{\fa\in\sigma_z}\min\{h_\el,\kappa^{-1}\}\norm{j}^2_\fa\right)^{1/2}
\end{equation}
 take, for every $\fa\in\sigma_z$, the function $\vphi=j\bub_\fa$, which yields
\begin{equation*}
 \begin{aligned}
&\dual{\interp f-\Op(U)}{\vphi}
=\sum_{\el\subset\omega_\fa}\int_\el r\vphi +\int_\fa j^2\bub_\fa
\end{aligned}
\end{equation*}
Exploiting \eqref{fb-std-L2}--\eqref{fb-std-H1} we derive
\begin{equation*}
 \begin{aligned}
\norm{j}_\fa^2&\lesssim \int_\fa j^2\bub_\fa\leq
\abs{\dual{\interp f-\Op(U)}{\vphi}}
+\sum_{\el\subset\omega_\fa}\abs{\int_\el r\vphi}
\\
&
\lesssim \tnorm{\interp f-\Op(U)}_{*;\omega_z}\tnorm{\vphi}_{\omega_\fa}
\\
&\qquad+ 
\left(\sum_{\el\subset\omega_\fa}\min\{h_\el,\kappa^{-1}\}^2\norm{r}^2_\el\right)^{1/2}
\max\{h_\fa^{-1},\kappa\}\norm{\vphi}_{\omega_\fa}
\\
&\lesssim\norm{j}_\fa\max\{h_\fa^{-1}, \kappa\}^{1/2}\tnorm{\interp f-\Op(U)}_{*;\omega_z}
\\
&\qquad+
\norm{j}_\fa\max\{h_\fa^{-1}, \kappa\}^{1/2}\left(\sum_{\el\subset\omega_\fa}\min\{h_\el,\kappa^{-1}\}^2\norm{r}^2_\el\right)^{1/2}
\end{aligned}
\end{equation*}
which combined with \eqref{ind-gtrsim-el-res} gives \eqref{ind-gtrsim-fa-res}. 
\end{proof}

Note that, differently from \eqref{alt-res}, the data term appears also in the face-supported part, and the approximation of $f$ in the element residual is a polynomial of first degree. 

\section{Conclusions}
\label{S:conc}
We derived a robust a posteriori error estimator for the Galerkin approximation of the reaction-diffusion equation \eqref{prob:rd} with conforming finite elements. The estimator together with the corresponding oscillation provide global upper and local lower bounds for the error in the energy norm. The definition of the oscillation follows the new approach of Kreuzer and Veeser \cite{Kreuzer.Veeser:16} and is robustly bounded by the error.  The indicators are computable under moderate assumptions on $f$ and can be used to drive an adaptive algorithm. Yet, the different structure of the oscillation and therefore of the jump and element residuals calls for new ideas for the proof of robust optimal rates of convergence.

\section*{Acknowledgements}
The first author was founded by the DFG under grant VE 397/1-1 AOBJ:612415 during her stay with the second author's group. The authors are grateful to A.~Veeser and C.~Kreuzer for making their manuscript available. 

\bibliographystyle{siam}

\def\cprime{$'$} \def\cprime{$'$}

\end{document}